\documentclass[12 pt]{amsart}
\usepackage[left=1in, right=1in, top=1in]{geometry}
\pdfoutput=1

\usepackage{amsmath}
\usepackage{mathtools}
\usepackage{color}
\usepackage{amsfonts}
\usepackage{enumerate}
\usepackage{amsthm}
\usepackage{chngcntr}
\usepackage{enumitem}
\usepackage[numbers]{natbib}
\usepackage{array}
\usepackage{titlesec}
\usepackage{hyperref}

\expandafter\def\expandafter\normalsize\expandafter{%
\normalsize
}

\titleformat{\section}
  {\normalfont\fontsize{15}{15}\bfseries}{\thesection}{1em}{}

\newtheorem{theorem}{Theorem}[section]
\newtheorem{corollary}[theorem]{Corollary}
\newtheorem{lemma}[theorem]{Lemma}
\newtheorem{prop}[theorem]{Proposition}
\newtheorem{remark}[theorem]{Remark}

\newtheorem{defn}[theorem]{Definition}
\newtheorem{ex}[theorem]{Example}
 
\DeclareMathOperator{\nab}{\mathbb{\nabla}}
\DeclareMathOperator{\ric}{Ric}
\DeclareMathOperator{\hess}{Hess}
\DeclareMathOperator{\Lie}{\mathcal{L}}

\newcommand{\bigslant}[2]{{\raisebox{.2em}{$#1$}\left/\raisebox{-.2em}{$#2$}\right.}}

\makeatletter
\renewenvironment{proof}[1][\proofname]{\par
  \pushQED{\qed}%
  \normalfont \topsep6\p@\@plus6\p@\relax
  \trivlist
  \item[\hskip\labelsep
        \itshape
    #1\@addpunct{.}]\mbox{}\\*
}{%
 \popQED\endtrivlist\@endpefalse
}
\makeatother

\begin{document}

\title[Splitting Theorem, Topology, and Nonnegative N-Bakry \'Emery Ricci Curvature]{\vspace{-2cm} The Splitting Theorem and Topology of Noncompact Spaces with Nonnegative N-Bakry \'Emery Ricci Curvature}
\author{Alice Lim}
\address{215 Carnegie Building\\
Dept. of Math, Syracuse University\\
Syracuse, NY, 13244.}
\email{awlim100@syr.edu}
\urladdr{https://awlim100.expressions.syr.edu}

\maketitle

\begin{abstract} In this paper, we generalize topological results known for noncompact manifolds with nonnegative Ricci curvature to spaces with nonnegative $N$-Bakry \'Emery Ricci curvature. We study the Splitting Theorem and a property called the geodesic loops to infinity property in relation to spaces with nonnegative $N$-Bakry \'Emery Ricci Curvature. In addition, we show that if $M^n$ is a complete, noncompact Riemannian manifold with nonnegative $N$-Bakry \'Emery Ricci curvature where $N>n$, then $H_{n-1}(M,\mathbb{Z})$ is $0$.
\end{abstract}

\maketitle

\section{Introduction}

One of the themes of Riemannian geometry is analyzing the topological implications of a manifold admitting a metric with a curvature constraint. In 1976, Yau proved that if $M^n$ is a complete, noncompact manifold with $\ric>0$, then $H_{n-1}(M,\mathbb{R})=0$ \cite{Yau}. In 2000, Shen-Sormani generalized this to show that such a space has $H_{n-1}(M,\mathbb{Z})=0$ by studying topological properties like the loops to infinity property \cite{ShSo}. In this paper, we will generalize these results to that of Riemannian manifolds with non-negative and positive Bakry-\'Emery Ricci curvature. Our results for positive curvature are optimal in the sense that none of the assumptions can be removed (See Examples \ref{ex:flz_homology_equals_z}, \ref{ex:spherecrossreal} and \ref{ex:hyperbolicspace}). 
\vspace{1em}

Riemannian manifolds with smooth positive density function $e^{-f}$ were first studied by Lichnerowicz in 1971 \cite{Lich}. Bakry and \'Emery studied this further in order to study diffusion processes \cite{BE}. More recently, Bakry-\'Emery Ricci tensors have been studied in optimal transport, Ricci flow, and general relativity. Qian proved in \cite{ZhongminQian} that Myers' Theorem holds for gradient $N$-Bakry \'Emery Ricci curvature. In \cite{Lott}, Lott gives topological consequences to nonnegative and positive Bakry-\'Emery Ricci curvature, as well as relations between the Bakry-\'Emery Ricci curvature bounded below and measured Gromov-Hausdorff limits. In \cite{WeiWylie}, Wei-Wylie proved Bakry-\'Emery Ricci curvature versions of the comparison theorems and the volume comparison theorem.  Fang-Li-Zhang in \cite{FLZ}, Khuri-Woolgar-Wylie in \cite{KhWoWy}, Munteanu-Wang in \cite{MunteanuWang}, and Wylie in \cite{Wy} also prove different versions of the Splitting Theorem for nonnegative Bakry-\'Emery Ricci curvature. We define the $N$-Bakry-\' Emery Ricci tensors as follows: 

\begin{defn}\label{defn:ricx}
Let $X$ be a vector field on $(M^n,g)$, a Riemannian manifold. The $N$-Bakry-\' Emery tensor is $$\ric_X^N:=\ric+\displaystyle\frac{1}{2}\Lie_X g-\frac{1}{N-n}X^*\otimes X^*$$

where $\Lie_X g$ is the Lie derivative of $g$ with respect to $X$, and \begin{align}
X^*  : T_p M & \rightarrow \mathbb{R} \nonumber \\
 Y & \mapsto g(X,Y)\nonumber.
\end{align}

If $X=\nabla \phi$ where $\phi:M\to\mathbb{R}$ is a smooth function, the $N$-Bakry-\' Emery Ricci tensor is  $$\ric_\phi^N:=\ric+\hess\phi-\displaystyle\frac{1}{N-n}d\phi\otimes d\phi.$$

If $X=\nabla \phi$ and $N=\infty$, then we denote $\ric_\phi:=\ric_\phi^\infty=\ric+\hess\phi$.
\end{defn}

\begin{remark}
Note that $\ric_X^N$ is a generalization of $\ric_\phi^N$ because if $X=\nab \phi$, then $\ric_X^N=\ric_\phi^N$. Similarly, we call $\ric_\phi^N$ a generalization of $\ric$ because if $\phi$ is constant, then $\ric_\phi^N=\ric$. 
\end{remark}

The constant $N$ is also called the synthetic dimension. According to Lott in \cite{Lott}, if $M$ is compact and satisfies $\ric_f^N>0$, then the warped product metric on $M\times S^{n-N}$ satisfies $\ric>0$. When $N=\infty$, the $\infty$-quasi Einstein equation is the Ricci soliton equation. When $N<0$, Wylie-Woolgar study $\ric_f^N$ in the context of Lorentzian scalar-tensor gravitational theories in cosomology in \cite{WylieWoolgar}. Milman \cite{Milman}, Ohta \cite{Ohta}, and Wylie \cite{Wy} also give descriptions of the condition $\ric_X^N$ bounded above with $N<0$.

\vspace{1em}

Next, we recall the definition of a line.
\vspace{1em}

\begin{defn}
A line is a geodesic $\gamma: (-\infty,\infty)\rightarrow M$, which is minimizing between any two points. See Figure \ref{fig:line}.
\end{defn}
\vspace{1em}

\begin{figure}[h!]
\includegraphics [width = .55\textwidth, height=.1\textheight]{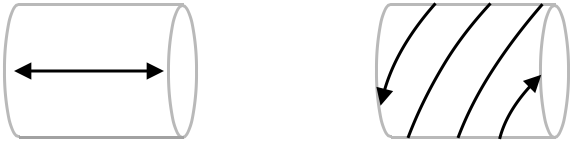}
\caption{The left image depicts a line on a cylinder. The right image does not depict a line on a cylinder because the curve is not minimizing.}
\label{fig:line}
\end{figure}     

In Section \ref{section:splitting}, we will review the Splitting Theorems which we will use throughout the paper. In particular, we will use a criterion by Fang-Li-Zhang to prove a new version of the Splitting Theorem.

\begin{prop}\label{prop:FLZ}
If $\ric_\phi\geq 0$ and $\nabla\phi\rightarrow 0$ at $\infty$, then the Splitting Theorem holds. 
\end{prop}

In \cite[Theorem~1.1]{FLZ}, Fang-Li-Zhang prove that if $(M,g)$ is a complete connected Riemannian manifold with $\ric_\phi\geq0$ for $\phi\in C^2(M)$ bounded above uniformly on $M$, then the Splitting Theorem holds. In \cite[Remark~3.1]{FLZ}, Fang-Li-Zhang note that their proof only relies on the condition $\displaystyle\lim_{\rho\rightarrow\infty}\frac{1}{\rho^2}\int_0^\rho \phi(\gamma(s))ds\leq 0$, where $\rho$ is the distance function and $\gamma$ is any ray. Thus, if $(M,g)$ is a complete connected Riemannian manifold with $\ric_\phi^\infty\geq 0$ and if $\phi$ satisfies the condition $\displaystyle\lim_{\rho\rightarrow\infty}\frac{1}{\rho^2}\int_0^\rho \phi(\gamma(s))ds\leq 0$, where $\rho$ is the distance function and $\gamma$ is any ray, then the Splitting Theorem holds \cite[Theorem~1.1, Remark~3.1]{FLZ}.

The main tool used by Shen and Sormani in \cite{ShSo} is the Cheeger-Gromoll Splitting Theorem, which states that if $\ric\geq 0$ and $M$ contains a line, then $M$ is isometric to a product metric, $\mathbb{R}^k\times N$, where $N$ doesn't contain any lines \cite{ChGr}.
\vspace{1em}

Using versions of the Splitting Theorem for $\ric_X^N$ and following an argument by Carron and Pedon in \cite{CarPed}, and Shen and Sormani in \cite{ShSo}, we will prove the following two theorems:

\begin{theorem}\label{thm:1}
Let $M^n$ be complete and noncompact. 

\begin{enumerate}
    \item If $\ric_X^N>0$ for $N>n$, then $H_{n-1}(M,\mathbb{Z})=0$.
    \item If $\ric_\phi^N > 0$ with $\phi<K$ for some $k\in\mathbb{R}$ and $N\leq 1$ , then $H_{n-1}(M,\mathbb{Z})=0$.
    \item If $\ric_\phi^\infty> 0$ with $\nabla\phi\rightarrow 0$ at $\infty$, then $H_{n-1}(M,\mathbb{Z})=0$.
    
\end{enumerate}
\end{theorem}

\begin{theorem}\label{thm:2} Let $M^n$ be complete and noncompact. 
\begin{enumerate}
    \item If $\ric_X^N\geq0$ for $N>n$, then $H_{n-1}(M,\mathbb{Z})=0\text{ or }\mathbb{Z}$.
    \item If $\ric_\phi^N \geq 0$ with $\phi<K$ for some $k\in\mathbb{R}$ and $N< 1$ , then $H_{n-1}(M,\mathbb{Z})=0\text{ or }\mathbb{Z}$.
    \item If $\ric_\phi^N \geq 0$ with $|\phi|<K$ for some $k\in\mathbb{R}$ and $N=1$ , then $H_{n-1}(M,\mathbb{Z})=0\text{ or }\mathbb{Z}$.
    \item If $\ric_\phi^\infty\geq 0$ with $\nabla\phi\rightarrow 0$ at $\infty$, then $H_{n-1}(M,\mathbb{Z})=0$ or $\mathbb{Z}$.
    \end{enumerate}
\end{theorem}

Next, we introduce the notion of a loop being homotopic to another loop along a ray, the loops to infinity property, and our main result regarding the loops to infinity property.

\begin{defn}
Given a ray $\gamma$ and a loop $C:[0,L]\rightarrow M$ based at $\gamma(0)$, we say that a loop $\widetilde{C}:[0,L]\rightarrow M$ is homotopic to $C$ along $\gamma$ if there exists $r>0$ with $\widetilde{C}(0)=\widetilde{C}(L)=\gamma(r)$ and the loop, constructed by joining $\gamma$ from $0$ to $r$ with $C$ from $0$ to $L$ and then with $\gamma$ from $r$ to $0$ is homotopic to $C$, in $\pi_1(M,\gamma(0))$.
\end{defn}

\begin{defn}
An element $h\in\pi_1(M,\gamma(0))$ has the geodesic loops to infinity property along $\gamma$ if for any $A\subset M$ compact, there exists a loop $\widetilde{C}\subset M\setminus A$ which is homotopic to a representative loop, C of $h$ along $\gamma$.
\end{defn}

\begin{theorem}\label{theorem:mainlemmashowsupinthistheorem}
Let $M^n$ be a complete, noncompact Riemannian manifold, and suppose one of the following holds:
\begin{enumerate}
    \item $\ric_X^N\geq 0$ with $N>n$.
    \item $\ric_\phi^N\geq 0$ with $N=\infty$, $\phi$ bounded above.
    \item $\ric_\phi^N\geq 0$ with $N\leq 1$ and $\phi$ bounded above.
    \item $\ric_\phi^N\geq 0$ with $N=\infty$, $\nabla\phi\rightarrow 0$ at $\infty$.
\end{enumerate}
\vspace{1em}
Then, 
\begin{enumerate}[label=(\roman*)]
\item If $g\in \pi_1(M)$, then either $g$ or $g^2$ has the geodesic loops to infinity property.
\item If there exists $g\in \pi_1(M)$ which does not satisfy the loops to infinity property along a given ray $\gamma$, then for all $h\in \pi_1(M,\gamma(0))$, $h$ must satisfy 
$$h_*(\widetilde{\gamma}'(t))=\pm\widetilde{\gamma}'(t).$$
Also, $M$ must have a split double cover which lifts $\gamma$ to a line.
\end{enumerate}
\end{theorem}

Next, we use various versions of the Splitting Theorem and we follow the proof of Shen-Sormani \cite{ShSo}, to classify the $n-1$ homology with $G$ coefficients, where $G$ is an Abelian group, given certain Bakry-\'Emery Ricci curvature constraints.

\begin{theorem}\label{thm:main theorem with Abelian group}

Let $M^n$ be a complete noncompact manifold with either of the following: 
\begin{enumerate}
    \item $\ric_X^N\geq 0$ with $N>n$.
    \item $\ric_\phi^N\geq 0$ with $\phi$ bounded above and $N\leq 1$ or $N=\infty$.
    \item $\ric_\phi^N\geq 0$ with $N=\infty$ and $\nabla\phi\rightarrow 0$ at $\infty$.
\end{enumerate}

Then we have the following cases:
\begin{enumerate}[label=(\roman*)]
    \item

    If $M^n$ has two or more ends and $G$ is an Abelian group, then 
    \[
  H_{n-1}(M,G) =
  \begin{cases}
                                   G & \text{if $M$ is orientable} \\
                                  \ker(G\overset{\times2}\to G) & \text{if $M$ is not orientable.}
  \end{cases}
\]
   
    \item If $M^n$ is one-ended with the loops to infinity property, then $H_{n-1}(M,G)=0$.
    \item If $M^n$ is one-ended and doesn't have a ray with the loops to infinity property, and G is an Abelian group, then 
     \[
  H_{n-1}(M,G) =
  \begin{cases}
                                   G & \text{if $M$ is not orientable} \\
                                  \ker(G\overset{\times2}\to G) & \text{if $M$ is orientable.}
  \end{cases}
\]

\end{enumerate}
\end{theorem}

In Section \ref{section:splitting} and \ref{section:homology}, we will give examples to show that our results in Theorem \ref{thm:1} are optimal. In Example \ref{ex:spherecrossreal}, we construct an example that satisfies the following: $\ric_\phi^N>0$ for $N\leq 1$ or $N=\infty$, $\phi$ is unbounded, and $H_{n-1}(M,\mathbb{Z})=\mathbb{Z}$. In Example \ref{ex:hyperbolicspace}, we will give an example where $\ric_\phi^N>0$ for $N=\infty$, $\nabla\phi$ doesn't converge to $0$ at $\infty$, and the Splitting Theorem does not hold. Finally, in Example \ref{ex:flz_homology_equals_z}, we will construct an example where $\ric_\phi^\infty>0$, $\nabla\phi$ is bounded, $\nabla\phi$ doesn't converge to 0 at $\infty$, and $H_{n-1}(M,\mathbb{Z})=\mathbb{Z}$. 
\vspace{1em}

In Section \ref{section:splitting}, we will review the Splitting Theorems for $\ric_X^N\geq0$ using work by Wylie in \cite{Wy}, Khuri-Woolgar-Wylie in \cite{KhWoWy}, Munteanu-Wang in \cite{MunteanuWang}, and Fang-Li-Zhang in \cite{FLZ}. 

\vspace{1em} 

In Section \ref{section:homology}, we review the proofs of Shen-Sormani \cite{ShSo} and Carron-Pedon \cite{CarPed} to classify the $n-1$ integral homology of non-compact manifolds with nonnegative Ricci curvature. Then, we prove the orientable case of Theorem \ref{thm:2} using Carron-Pedon's method of proof.
\vspace{1em}

In Section \ref{section:loopstoinfinity}, we give a sketch of Sormani's Line Theorem. We also prove our main lemma, Lemma \ref{lemma: mainresult}. We use this to prove Theorem \ref{theorem:mainlemmashowsupinthistheorem} and Theorem \ref{thm:2} in the non-orientable case. In Theorem \ref{thm:main theorem with Abelian group}, we also prove a more general version of Theorem \ref{thm:2} for $H_{n-1}(M,G)$ where $G$ is an Abelian group. 
\vspace{1em}

\section{The Splitting Theorem for Spaces with Nonnegative Bakry \'Emery Ricci Curvature}\label{section:splitting}
In this section we will describe the different versions of the Cheeger-Gromoll Splitting Theorem for the $N$-Bakry \'Emery Ricci curvature.
\vspace{1em} 

The $\ric_X^N\geq 0$ assumption becomes a weaker hypothesis as $N$ increases. $N >n$ is our strongest premise and the Splitting Theorem holds with no further assumptions, as shown by Khuri-Woolgar-Wylie in \cite[Theorem~2]{KhWoWy} and Fang-Li-Zhang in \cite[Theorem~1.3]{FLZ}. $N<1$ or $N=\infty$ is a weaker premise and the Splitting Theorem does not hold in general; however, Wylie showed that including the additional assumptions $X=\nabla \phi$ and $\phi<K$ for $K$ constant gives a splitting \cite[Corollary~1.3]{Wy}. Munteanu-Wang also showed in \cite[Theorem~1.6]{MunteanuWang} that if $N=\infty$ and $X=\nabla \phi$ where $\phi$ has linear growth with a weighted entropy condition, then the Splitting Theorem holds or $M$ is connected at infinity. If $N=1$ the Splitting Theorem does not hold, even when $\phi$ is bounded. However, if $X=\nabla \phi$ with $\phi<K$, then Wylie showed in \cite[Theorem~1.2]{Wy} that there is a more general warped product splitting. Here, we say that $(M,g)$ has a warped product splitting if $M$ is diffeomorphic to $\mathbb{R}\times L$ where $L$ is an $(n-1)$-dimensional manifold and there exists $u:\mathbb{R}\rightarrow \mathbb{R}^+$ such that $g=dr^2+u^2(r)g_0$ for a fixed metric $g_0$. We call $g$ a warped product over $\mathbb{R}$.
\vspace{1em} 

Using a remark by Fang-Li-Zhang, we can show that if $\ric_\phi\geq 0$ and $\nabla \phi\rightarrow 0$ at $\infty$, the Splitting Theorem holds \cite[Remark~3.1]{FLZ}. This is interesting because unlike the generalization of Myers' Theorem, which says if $\ric_\phi^\infty\geq \lambda g>0$, and if $|\nabla \phi|$ is bounded, then $M^n$ is compact \cite{LopezRio}, we need $\nabla\phi\rightarrow 0$ at $\infty$ rather than $|\nabla\phi|$ bounded in order for the Splitting Theorem to hold for $\ric_\phi^\infty$.
\vspace{1em}

The following table summarizes the known versions of the Splitting Theorem for $\ric_X^N\geq 0$:
\vspace{1em}
\begin{center}
\fbox{
  \parbox{35em}{

\begin{center}
    \underline{If $\ric_X^N\geq 0$, then:}
\end{center}
\begin{equation*}
\begin{array}{lcl}
             N>n & \Rightarrow & \text{Splitting Theorem}\text{, \cite{KhWoWy},\cite{FLZ}}\\
  N<1\text{, }X=\nabla\phi\text{, } \phi<K & \Rightarrow &\text{Splitting Theorem}\text{, \cite{Wy}} \\
  N=\infty\text{, }X=\nabla\phi\text{, }\phi<K & \Rightarrow & \text{Splitting Theorem}\text{, \cite{FLZ}}\\
  N=1\text{, }X=\nabla\phi\text{, }\phi<K & \Rightarrow & \text{Warped Product Splitting}\text{,  \cite{Wy}}\\
  N=\infty\text{, } X=\nabla\phi\text{, }\nabla\phi\rightarrow 0\text{ at }\infty & \Rightarrow & \text{Splitting Theorem, \cite{FLZ}}
       \end{array}
\end{equation*}
  }
}  
\end{center}
\vspace{1em}

Before we state our next theorem, we will show that if $\ric_\phi\geq 0$ and $\nabla\phi\rightarrow 0$, then the Splitting Theorem holds. 
\vspace{1em}

\begin{proof}[Proof of Proposition \ref{prop:FLZ}]

Suppose $\nabla\phi\rightarrow 0$ at $\infty$. Then, for each $\varepsilon>0$, there exists $R>0$ such that for all $x\in M\setminus\overline{B(\gamma(0),R)}$, $|\nabla\phi|(x)< \varepsilon$. 
\vspace{1em}

Let $\gamma(t)$ be a unit speed ray. Then, 
\[
\phi'(\gamma(t))=\langle \nabla\phi,\dot\gamma\rangle\leq|\nabla\phi|\leq\varepsilon.
\]

 After integrating, for the same $\varepsilon>0$ and $R>0$, we get $\phi(\gamma(t))<\varepsilon t+C$, where $C$ is a constant. Then,
\vspace{1em}

\begin{align*}
\displaystyle\lim_{\rho\rightarrow\infty}\frac{1}{\rho^2}\int_0^\rho \phi(\gamma(s))ds&=\displaystyle\lim_{\rho\rightarrow\infty}\frac{1}{\rho^2}\left(\int_0^R \phi(\gamma(s))ds+\int_R^\rho \phi(\gamma(s))ds\right)\\
&<\displaystyle\lim_{\rho\rightarrow\infty}\left(\frac{1}{\rho^2}\left(\int_0^R \phi(\gamma(s))ds\right)+\frac{\varepsilon }{2}+\frac{C}{\rho}-\frac{\varepsilon R^2}{2\rho^2}-\frac{CR}{\rho^2}\right)\\
&=\displaystyle\frac{\varepsilon}{2}.
\end{align*}

Letting $\varepsilon\rightarrow 0$, we get
\begin{align*}
\displaystyle\lim_{\rho\rightarrow\infty}\frac{1}{\rho^2}\int_0^\rho \phi(\gamma(s))ds\leq 0.
\end{align*}
Thus, by \cite[Remark~3.1]{FLZ}, the Splitting Theorem holds.
\end{proof}
\vspace{1em}

Now, we will give an example where $\ric_\phi> 0$, $\nabla\phi$ is bounded, $\displaystyle\lim_{\rho\rightarrow\infty}\frac{1}{\rho^2}\phi(\gamma(s))ds$ is nonzero and finite, the Splitting Theorem doesn't hold, and $H_{n-1}(M,\mathbb{Z})=\mathbb{Z}$, rather than $0$.

\begin{ex}\label{ex:flz_homology_equals_z}
Let $M=\mathbb{R}\times S^{n-1}$, where our metric is $g=dr^2+\rho^2(r)g_N$. We wish to construct $\rho(r)$ and $\phi(r)$ such that $\ric_\phi>0$ everywhere and $\phi(r)$ and $\rho(r)$ are smooth.\\ 

Let $\rho$ be a function such that \[\begin{cases} 
\rho>0 & everywhere\\
|\dot\rho|<1 & everywhere\\
-C<\ddot\rho<0 & |r|>A\\
\end{cases} \]
where $A$ and $C$ are constants. Figure \ref{Figure:1} is an example of what $\rho$ might look like:

\begin{figure}[ht]
\centering

\includegraphics[width=0.4\textwidth]{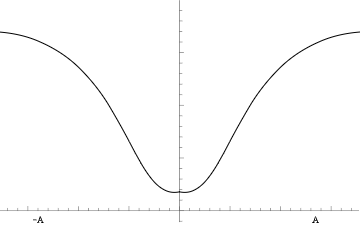}
\caption{$\rho(r)$}
\label{Figure:1}
\end{figure}

\vspace{1em}
Later in the example, we will consider $\varepsilon\rho$ where $\varepsilon>0$, so the space will look like a cylinder with a small dip around $0$. We proceed with our calculations:
\vspace{1em}

Let $V$ be a vector $TS^{n-1}$. Given our metric, $\ric_\phi(\frac{\partial}{\partial r},\frac{\partial}{\partial r})=-(n-1)\frac{\ddot\rho}{\rho}+\ddot \phi$ and $\ric_\phi (V,V)=(n-2)(1-\dot\rho^2)-\rho\ddot\rho+\dot \phi\dot\rho$ (See \cite{PP}, page 69).\\

On $|r|<A$, there exists a smooth function, $\alpha(r)$, larger than $(n-1)\displaystyle\frac{\ddot\rho}{\rho}$ on $|r|<A$ such that $\alpha(A)=\alpha(-A)=0$, because $-(n-1)\displaystyle\frac{\ddot\rho}{\rho}(\pm A)<0$.\\

Let $\phi$ be a function such that $\phi''(r)=\alpha(r)$.Then, $\ric_\phi(\frac{\partial}{\partial r},\frac{\partial}{\partial r})>0$ everywhere.\\

Now, consider $\varepsilon\rho$ in place of $\rho$ where $\varepsilon>0$.\\

Then we still get $\ric_\phi(\frac{\partial}{\partial r},\frac{\partial}{\partial r})=-(n-1)\frac{\ddot\rho}{\rho}>0$.\\

$\ric_\phi(V,V)=(n-2)(1-\varepsilon^2\dot\rho^2)-\varepsilon^3\rho\ddot\rho+\dot \phi\varepsilon^2\dot\rho\rho$. Letting $\varepsilon\rightarrow 0$, $\ric_\phi(V,V)\rightarrow n-2>0$ since $\dot \phi$ is bounded.\\ 

Finally, we have $\ric_\phi>0$ everywhere.\\

On $|r|>A$, $\phi''(r)=0$, which means $\phi(r)=Br+E$ on $|r|>A$.

Thus, $\displaystyle\lim_{t\rightarrow \infty}\frac{1}{t^2}\displaystyle\int_0^{t} \phi\circ\gamma dr=\displaystyle\lim_{t\rightarrow \infty}\displaystyle\frac{1}{t^2}\int_0^A \phi\circ\gamma(r)dr+\displaystyle\frac{1}{t^2}\int_A^t (Br+E)dr=\displaystyle\frac{B}{2}$. 
\end{ex}
\vspace{1em}

\section{Codimension One Homology of Noncompact Spaces with Nonnegative Bakry \'Emery Ricci Curvature}\label{section:homology}
In this section, we will generalize the proofs of Shen-Sormani \cite{ShSo} and Carron-Pedon \cite{CarPed} in order to classify $H_{n-1}(M,\mathbb{Z})$ for noncompact spaces with nonnegative Bakry \'Emery Ricci curvature. We will also give examples to show that our results are optimal.
\vspace{1em}

First, we define what it means to be a one-ended manifold.
\vspace{1em}

\begin{defn}
A manifold is one-ended if given any compact set $A\subset M$, $M\setminus A$ has only one unbounded connected component. See Figure \ref{fig:oneendedmanifold}.
\end{defn}
\vspace{1em}

\begin{figure}[ht]
    \centering
    \includegraphics[width=0.55\textwidth, height=.1\textheight]{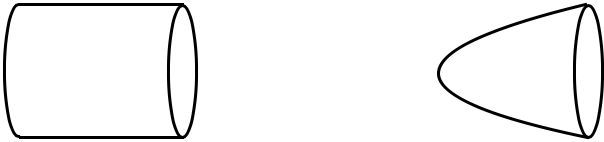}
    \caption{The left image shows a two-ended manifold and the right image shows a one-ended manifold.}
    \label{fig:oneendedmanifold}
\end{figure}

Next, we review a proposition by Carron and Pedon, which we will use to prove Theorem \ref{thm:2} in the orientable case. 
\vspace{1em}
\begin{prop}\cite[Proposition~5.2]{CarPed}\label{prop:carronandpedon}
If $M^n$ is an orientable open manifold having one end, and if every twofold normal covering of $M$ also has one end, then $H_{n-1}(M,\mathbb{Z})=0$.

\end{prop}
\vspace{1em}

We will now prove Theorem \ref{thm:2} for $M^n$ orientable. To do this, we will follow \cite[Proposition~5.3]{CarPed} in using Proposition \ref{prop:carronandpedon}. Later, we will prove Theorem \ref{thm:2} in the $M^n$ non-orientable case by using the loops to infinity property.


\begin{proof}[Proof of Theorem \ref{thm:2} (orientable case)]
Suppose $M$ is one-ended and suppose every double cover of $M$ is one-ended. By Proposition \ref{prop:carronandpedon}, $H_{n-1}(M,\mathbb{Z})=0$. 
\vspace{1em} 


Suppose $M$ is one-ended and there exists a double cover, $\widetilde{M}$, which is two-ended. Then $\widetilde{M}$ splits isometrically as $\widetilde{L}\times\mathbb{R}$, where $\widetilde{L}$ is compact by the Splitting Theorem. Let $h$ be the nontrivial deck transformation acting on $\widetilde{M}$.  Then $H_{n-1}(M,\mathbb{Z})= H_{n-1}(M,\mathbb{Z})=H_{n-1}\bigg(\displaystyle\frac{\widetilde{L}\times\mathbb{R}}{\langle h\rangle},\mathbb{Z}\bigg)$. Since $M$ is one-ended, $M$ is orientable if and only if $\widetilde{L}$ is non-orientable. Thus, $H_{n-1}(M,\mathbb{Z})=H_{n-1}\bigg(\displaystyle\frac{\widetilde{L}\times\mathbb{R}}{\langle h\rangle},\mathbb{Z}\bigg)=0$.
\vspace{1em}

Suppose $M$ is two ended. By the Splitting Theorem, $M$ is isometric to $L\times \mathbb{R}$ where $L$ is compact and has the same orientability as $M$. Then, since $M$ is orientable, $H_{n-1}(M,\mathbb{Z})=H_{n-1}(N,\mathbb{Z})=\mathbb{Z}$.
\end{proof}
\vspace{1em}

 In the following example, we will give a space and metric where $\ric_\phi^N>0$ for $N=\infty$ and $N\leq 1$, $\phi$ is unbounded, and $H_{n-1}(M,\mathbb{Z})=\mathbb{Z}$.
\begin{ex}\label{ex:spherecrossreal}
Consider $M= S^{n-1}\times \mathbb{R}$.\\

Let $\phi: M^n\rightarrow \mathbb{R}$, $\phi(r)=r^2$.\\

In the $S^{n-1}$ direction, $\ric>0$ and $\hess\phi=0$. In the $\mathbb{R}$ direction, $\ric=0$ and $\hess>0$.\\ 

If $N=\infty$ or $N\leq 1$, then $-\displaystyle\frac{1}{N-n}\nabla\phi^*\otimes \nabla\phi^*\geq 0$.\\ 

Therefore, $\ric_\phi^N>0$ for $N=\infty$ and $N\leq 1$. However, $H_{n-1}(S^{n-1}\times \mathbb{R},\mathbb{Z})=H_{n-1}(S^{n-1},\mathbb{Z})=\mathbb{Z}$. Notice that $\phi$ is unbounded.

\end{ex}
\vspace{1em}

Observe that in Example \ref{ex:spherecrossreal}, the Splitting Theorem does hold. The next example, which Wei-Wylie constructed in \cite[Example~2.2]{WeiWylie}, satisfies $\ric_\phi^\infty>0$ and $\displaystyle\lim_{\rho\rightarrow\infty}\frac{1}{\rho^2}\int_0^\rho \phi(\gamma(s))ds=\infty$, yet the Splitting Theorem does not hold. 

\begin{ex} \label{ex:hyperbolicspace}
Consider $M^n=\mathbb{H}^n$. Fix $p\in M$. Let $\phi(x)=(n-1)d(x,p)^2$, where $d(x,p)$ is the distance to $p$. Then, $\ric_\phi\geq 0$ and the Splitting Theorem does not hold, as in \cite[Example~2.2]{WeiWylie}. Also, if $\gamma(0)=p$, then\\

$\displaystyle\lim_{\rho\rightarrow\infty}\frac{1}{\rho^2}\int_0^\rho \phi(\gamma(s))ds$\\

$=\displaystyle\lim_{\rho\rightarrow \infty}\frac{1}{\rho^2}\int_0^\rho(n-1)d(\gamma(s),\gamma(0))^2ds$\\

$=\displaystyle\lim_{\rho\rightarrow\infty}(n-1)\displaystyle\frac{\rho}{3}=\infty$.
\end{ex}
\vspace{1em}

\section{Geodesic Loops to Infinity Property}\label{section:loopstoinfinity}
We will present the proof of Sormani's Line Theorem, which, along with the Splitting Theorem for $\ric_X^N\geq 0$, allows us to prove our main lemma, Lemma \ref{lemma: mainresult}. We then prove Theorem \ref{thm:2} in the non-orientable case. We are ready to present Sormani's Line Theorem.
\vspace{1em}

\begin{theorem}\cite[Theorem~1.7]{So}\label{thm: linetheorem}
If $M^n$ is a complete non-compact manifold which does not satisfy the geodesic loops to infinity property, then there is a line in its universal cover. 
\end{theorem}

\begin{proof}
Since $M^n$ is a complete, non-compact manifold, there exists a ray, $\gamma: [0,\infty)\rightarrow M^n$. Let $h\in\pi_1(M,\gamma(0))$ which does satisfy the loops to infinity property, and let $C$ be a representative of $h$ based at $\gamma(0)$. Because $h$ doesn't satisfy the loops to infinity property, there exists a compact set $A\subset M$ such that any loop homotopic to $C$ along $\gamma$ intersects $A$. Let $R>0$ such that $A\subset B_{\gamma(0)}(R)$. Let $\{r_i\}$ be a sequence such that $\displaystyle\lim_{i\to\infty}r_i=\infty$ and $r_i>R$ for all $i$. 
\vspace{1em}

Now, let $\widetilde{M}$ be the universal cover of $M$, and let $\pi:\widetilde{M}\rightarrow M$ be the covering map. Identifying loops in $\pi_1(M,\gamma(0))$ with deck transformations, let $\widetilde{\gamma}$ and $h\circ\widetilde{\gamma}$ be lifts of $\gamma$ starting at $\widetilde{\gamma}(0)$ and $h\circ\widetilde{\gamma}(0)$ respectively, and let $\widetilde{C}$ be the lift of $C$, starting at $\widetilde{\gamma}(0)$ and ending at $h\circ\widetilde{\gamma}(0)$. If $\widetilde{C_i}$ are minimal geodesics from $\widetilde{\gamma}(r_i)$ to $h\circ\widetilde{\gamma}(r_i)$, then, $C_i:=\pi(\widetilde{C}_i)$ is a loop based at $\gamma(r_i)$ which is homotopic to $C$ along $\gamma$.
\vspace{1em}

Let $L_i=L(\widetilde{C}_i)=L(C_i)=d_{\widetilde{M}}(\widetilde{\gamma}(r_i),h\circ\widetilde{\gamma}(r_i))$. For each $C_i$, there exists some $t_i\in [0,L_i]$ such that $C_i(t_i)\subset A$.
\vspace{1em}

Let $\widetilde{A}$ be the lift of $A$ to the fundamental domain in $\widetilde{M}$. For all $i\in \mathbb{N}$, there exists $h_i\in\pi_1(M,\gamma(0))$, so that $h_i\circ\widetilde{C}_i(t_i)\in\widetilde{A}$.
\vspace{1em}

Through some computational details which we will omit, (See \cite[Theorem~1.7]{So} for more details), $h_i\circ\widetilde{C}_i$ are minimal geodesics from $(t_i-(r_i-R))$ to $(t_i+(r_i-R))$ such that $h_i\circ\widetilde{C}_i(t_i)\in\widetilde{A}$. Letting $r_i\rightarrow \infty$, a subsequence of $h_{i*}\circ\widetilde{C}_i'(t_i)$ converges to a vector $\gamma_\infty'(0)$, based at $\gamma_\infty(0)$ in the closure of $\widetilde{A}$. Let $\gamma_\infty$ be the geodesic with these initial conditions. Then $\gamma_\infty$ runs from $\displaystyle\lim_{i\to\infty}t_i-(r_i-R)=-\infty$ to $\displaystyle\lim_{i\to\infty}t_i+(r_i-R)=\infty$. Thus, we have constructed a line, namely $\gamma_\infty$, in $\widetilde{M}$.

\end{proof}
The following corollary follows from \cite{ShSo} and the generalizations of the Splitting Theorem.

\begin{corollary}
Let $M^n$ be a complete, noncompact Riemannian manifold, and suppose one of the following holds:
\begin{enumerate}
    \item $\ric_X^N\geq 0$ with $N>n$.
    \item $\ric_\phi^N\geq 0$ with $N=\infty$, $\phi$ bounded above.
    \item $\ric_\phi^N\geq 0$ with $N\leq 1$ and $\phi$ bounded above.
    \item $\ric_\phi^N\geq 0$ with $N=\infty$, $\nabla\phi\rightarrow 0$ at $\infty$.
\end{enumerate}
\vspace{1em}
Then, 

\begin{enumerate}[label=(\roman*)]
\item If $D$ be a precompact subset of $M$ and $\partial D$ is simply connected, then $\pi_1(D)$ can only contain elements of order 2. 
\item If $D$ be a precompact subset of $M$ with smooth boundary, where $\gamma$ is a ray such that $\gamma(0)\in D$ and if $S$ be any connected component of $\partial D$ containing a point $\gamma(a)$, then the image of the inclusion map
$$i_* : \pi_1(S, \gamma(a))\rightarrow  \pi_1(Cl(D), \gamma(a))$$
is $N \subset \pi_1(Cl(D), \gamma(a))$ such that $\pi_1(Cl(D), \gamma(a))/N$ contains at most two
elements.

\end{enumerate}

\end{corollary}
\vspace{1em}

\begin{corollary}\label{corollary:geodesicloopstoinfinity}
Let $M^n$ be a complete, noncompact Riemannian manifold, and suppose one of the following holds:
\begin{enumerate}
    \item $\ric_X^N\geq 0$ with $N>n$, and there exists a point $p\in M$ such that $(\ric_X^N)_p>0$.
    \item $\ric_\phi^N\geq 0$ with $N=\infty$, $\phi$ bounded above, and there exists a point $p\in M$ such that $(\ric_\phi^N)_p>0$.
    \item $\ric_\phi^N\geq 0$ with $N\leq 1$ and $\phi$ bounded above, and there exists a point $p\in M$ such that $(\ric_X^N)_p>0$.
    \item $\ric_\phi^N\geq 0$ with $N=\infty$, $\nabla\phi\rightarrow 0$ at $\infty$, and there exists a point $p\in M$ such that $(\ric_\phi^N)_p> 0$.
\end{enumerate}
\vspace{1em}
Then, $M^n$ has the geodesic loops to infinity property.

\end{corollary}

\vspace{1em}

\begin{proof}
First, we will show that $M$ and its universal cover, $\widetilde{M}$, have no lines. Suppose for the sake of contradiction that $M$ contains a line. We saw earlier in the paper that each of the four premises gives us a version of the Splitting Theorem. Hence, $M=\mathbb{R}\times N$. However, $\ric_\phi^N(\frac{\partial}{\partial r},\frac{\partial}{\partial r})=0$, which is a contradiction, thus proving our claim.
\vspace{1em}

Ergo, by Sormani's Line Theorem, $M$ has the loops to infinity property.
\end{proof}
\vspace{1em}

Before we prove the next proposition, we will first show that there exist examples of Riemannian manifolds with $\ric_\phi^1\geq 0$ not satisfying loops to infinity property along a given ray $\gamma$ and universal cover which has a warped product splitting. 

\begin{ex} Let $\phi\in C^2$ be bounded with bounded first and second derivatives.  By \cite[Corollary~2.4]{Wy}, there exists $\lambda $ large enough so that $\ric_\phi^1\geq 0$ and $g=dt^2+e^{\frac{2\phi}{n-1}}S_\lambda^n$. Now consider $M=\bigslant{(\mathbb{R}\times S^n)}{G}$, where $G$ is the group generated by $h(t,x)=(a-t,-x)$ for any constant $a>0$. If we also assume $\phi(a-t)=\phi(t)$, then $h$ is an isometry and $(M,g,\phi)$ satisfies $\ric_\phi^1\geq 0$. $h$ does not have the loops to infinity property along $(t,0)=(-t,a)$, so $\bigslant{(\mathbb{R}\times S^n)}{G}$ satisfies all of the necessary properties.
\end{ex}
\vspace{1em}

Next, we recall the definition of a ray lying in the split direction and we state a theorem which we will use in our main result.

\begin{defn}

In a warped product splitting, $M^n=N\times \mathbb{R}$ with $g=e^{\frac{f(r)}{n-1}}g_N+dr^2$, we say that $\gamma$, a ray, lies in the split direction if $\gamma(r)=(x_0,r)$, where $x_0\in N$.
\end{defn}

\begin{theorem}\label{thm: wylabel12}\cite[Proposition~4.2]{Wy}
Consider a warped product metric of the form $g=dr^2+v^2(r)g_N$ where $v>0$ is bounded from above. Let $\gamma:(a,b)\rightarrow M$ be a unit speed minimizing geodesic in $M$ and write $\gamma(s)=(\gamma_1(s),\gamma_2(s))$, where $\gamma_1$ and $\gamma_2$ are projections in the factors $\mathbb{R}$ and $N$. Then:
\begin{itemize}
    \item[(1)] $\gamma_2$ is either constant or its image is a minimizing geodesic in $(N,g_N)$.
    \item[(2)] If $\gamma_2$ is not a constant and $\gamma$ is a line in $M$, then the image of $\gamma_2$ is a line in $N$.
\end{itemize}
\end{theorem}
\vspace{1em}

Next, we state remarks from \cite[Lemma~4.4]{Wy} and \cite[page~208, Remark~8]{ON} which we will use in the proof of Lemma \ref{lemma: mainresult}.
\vspace{1em}

\begin{remark} \cite[Lemma~4.4]{Wy} \label{remark:isometriessplit}
In the context of Theorem \ref{thm: wylabel2}, when $g=e^{\frac{2f(r)}{n-1}}g_N+dr^2$ and $N$ does not admit a line, it follows from Theorem \ref{thm: wylabel12} that if $h:M\rightarrow M$ is an isometry, then $h=h_1\times h_2$ where $h_1\in Isom(N)$ and $h_2\in Isom(\mathbb{R})$.
\end{remark}
\vspace{1em}

\begin{remark}\cite[page~208, Remark~8]{ON} \label{remark:ON}
Let $\gamma=(x(t),y(t))$ be a geodesic in the warped product, $M=N\times \mathbb{R}$, where the metric tensor is $g=e^{\frac{2f(r)}{n-1}}g_N+dr^2$. Then, the function $e^{\frac{4f(y(t))}{n-1}}|x'(t)|^2$ is a constant $C$.
\end{remark}
\vspace{1em}

The following theorem, which can be found in \cite{Wy}, is the Splitting Theorem for $\ric_\phi^1$. We will use this to prove our main lemma, Lemma \ref{lemma: mainresult}. 

\begin{theorem}\label{thm: wylabel2}\cite[Lemma~4.4]{Wy}
Suppose that $(M, g, \phi)$ satisfies $\ric_\phi^1\geq 0$ with $\phi$ bounded (above and below)
and contains a line. Then either the Cheeger Gromoll Splitting Theorem holds or $M$ is diffeomorphic to $N\times \mathbb{R}$ and $g=e^{\frac{2f(r)}{n-1}}g_N+dr^2$ where $\phi=f+f_N$ and $(N,g_N)$ does not admit a line.
\end{theorem}
\vspace{1em}

We are prepared to state our main result.
\vspace{1em}

\begin{lemma}\label{lemma: mainresult}
Let $(M,g,\phi)$ be a Riemannian manifold with $\ric_\phi^1\geq 0$ and $|\phi|\leq K$ for $K>0$. Suppose there exists $h\in\pi_1(M)$ which does not satisfy the geodesic loops to infinity property along a given ray $\gamma$. Then the lift $\widetilde{\gamma}$ of $\gamma$ is in the split direction, 

$$\widetilde{\gamma}(t)=(x(0),y(t))$$ and $$h_*(\widetilde{\gamma}'(t))=-\widetilde{\gamma}'(t).$$
\end{lemma}

See Figure \ref{fig:mainlemma} for an image representation of Lemma \ref{lemma: mainresult}.

\begin{proof}
Let $(\widetilde{M},\widetilde{g})$ be the universal cover of $M$. By Theorem \ref{thm: linetheorem}, there exists a line in $\widetilde{M}$. By Theorem \ref{thm: wylabel2}, we have the following cases: either $\widetilde{M}=N\times\mathbb{R}^k$ and $\widetilde{g}=g_N+g_{\mathbb{R}^k}$, or $\widetilde{M}=N\times\mathbb{R}$ with $\widetilde{g}=e^{\frac{2f(r)}{n-1}}g_N+dr^2$, where $N$ contains no lines.

\vspace{1em}
If $\widetilde{g}=g_N+g_{\mathbb{R}^k}$, then we have a product metric, so we can follow the proof of \cite[Proposition~1.9]{So} to obtain the desired conclusion. 
\vspace{1em}

Suppose $\widetilde{g}=e^{\frac{2f(r)}{n-1}}g_N+dr^2$, where $N$ contains no lines. Recall the setup of Theorem \ref{thm: linetheorem}. We know that there are minimal geodesics $\widetilde{C}_i$ running from $\widetilde{\gamma}(r_i)$ to $h\circ\widetilde{\gamma}(r_i)$. \cite[Theorem~1.7]{So}
\vspace{1em}

Let $p_N:\Tilde{M}\rightarrow N$ and $p_{\mathbb{R}}:\Tilde{M}\rightarrow \mathbb{R}$ be the projections onto the $N$ component and the $\mathbb{R}$ component, respectively. Let $\widetilde{C}_i(t)=(x_i(t),y_i(t))$, where $x_i(t):=p_N(\widetilde{C}_i(t))$ and $y_i(t):=p_\mathbb{R}(\widetilde{C}_i(t))$. We have $h_*(\widetilde{C}_i'(t_i))=h_*(x_i'(t_i),y_i'(t_i))=(h_{1*}\circ x_i'(t_i),h_{2*}\circ y_i'(t_i))$ for $t_i\in (0,L_i)$ as in Theorem \ref{thm: linetheorem}. The last equality follows from Remark \ref{remark:isometriessplit}.
\vspace{1em}

We have that $h_*(\widetilde{C}_i'(t_i))\rightarrow \gamma_\infty'(0)$, where we can define $\gamma_\infty'(0)=(x_\infty'(0),y_\infty'(0))$ with $x_\infty(t)=p_N(\gamma_\infty(t))$ and $y_\infty(t)=p_{\mathbb{R}}(\gamma_\infty(t))$. Now, $h_*(x_i'(t_i))\rightarrow x_\infty'(0)$ and $h_*(y_i'(t_i))\rightarrow y_\infty'(0)$. Thus, $\displaystyle\lim_{i\to\infty}|x_i'(t_i)|=\displaystyle\lim_{i\to\infty}|h_*(x_i'(t_i))|=|x_\infty'(0)|$.
\vspace{1em}

By Theorem \ref{thm: wylabel12}, since $\gamma_\infty(t)$ is a line, $x_\infty(t)$ is either the image of a line or constant. However, $N$ doesn't contain any lines, so $|x_\infty'(0)|=0$. Now, using Remark \ref{remark:ON}, $e^{\frac{4f(y(t_i))}{n-1}}|x_i'(t_i)|^2 = e^{\frac{4f(y(0))}{n-1}}|x_i'(0)|^2$. Since $|f|\leq K$ for some $K>0$, for all $t\in \mathbb{R}$,

$$|x_i'(t)|^2=e^{\frac{4f(y(t_i))-4f(y(t))}{n-1}}|x_i'(t_i)|^2\leq B|x_i'(t_i)|^2,\textrm{ where } B=e^{\frac{8K}{n-1}}.$$
\vspace{1em}

Then, $$\displaystyle\lim_{i\to\infty}|x_i'(0)|^2\leq B\lim_{i\to\infty}|x_i'(t_i)|^2=B\displaystyle\lim_{i\to\infty}|h_*(x_i'(t_i))|=B|x_\infty'(0)|=0.$$ 

So, we know that $$\displaystyle\lim_{i\to\infty}|x_i'(0)|^2=0.$$

\vspace{1em} We want to show that for any $t$, there exists $i_0\in \mathbb{N}$ such that for all $i\geq i_0$, $|y_i'(t)|$ is strictly positive. Suppose for the sake of contradiction that there exists some $t_1$ such that for all $i\geq i_0$, $y_i'(t_1)=0$. Then, since $\widetilde{C}_i(t)$ is unit speed, we have
$$|y_i'(t_1)|^2+e^{\frac{2f(t_1)}{n-1}}|x_i'(t_1)|^2=1,\text{ so for all } i\geq i_0,$$
$$|x_i'(t_1)|^2=e^{\frac{-2f(t_1)}{n-1}}(1-|y_i'(t_1)|^2)=e^{\frac{-2f(t_1)}{n-1}}.$$

As $i\rightarrow \infty$, $|x_i'(t_1)|\rightarrow 0$, however, $$0<e^{\frac{-2K}{n-1}}<e^{\frac{-2f(t_1)}{n-1}}<e^{\frac{2K}{n-1}},$$
which is a contradiction. Thus, for all $t$, there exists $i$ large enough so that $|y_i'(t)|$ is strictly positive.

In particular, since $|y_i'(t)|$ is never 0 in $\mathbb{R}$, $|y_i(t)|$ never changes direction, and so $$d_\mathbb{R}(y(r_i),h(y(r_i))= L(y_i(t))=\int_0^{L_i}|y_i'(t)|dt.$$
\vspace{1em}

Now, $\displaystyle\frac{d_\mathbb{R}(y(r_i),h(y(r_i)))}{L_i}=\displaystyle\frac{L(y_i(t))}{L_i}=\frac{\int_0^{L_i}|y_i'(t)|dt}{\int_0^{L_i}\sqrt{e^{\frac{2f(t)}{n-1}}|x_i'(t)|^2+|y_i'(t)|^2}dt}$.

Since $e^{\frac{2f(t)}{n-1}}|x_i'(t)|^2\leq e^{\frac{2K}{n-1}}|x_i'(t)|^2\rightarrow 0$, and $|y_i'(t)|^2=1-e^{\frac{2f(t)}{n-1}}|x_i'(t)|^2$, we get that \[\displaystyle\frac{\int_0^{L_i}|y_i'(t)|dt}{\int_0^{L_i}\sqrt{e^{\frac{2f(t)}{n-1}}|x_i'(t)|^2+|y_i'(t)|^2}dt}\geq \displaystyle\frac{\int_0^{L_i}(1-\varepsilon_i) dt}{\int_0^{L_i}\sqrt{e^{\frac{2K}{n-1}}|x_i'(t)|^2+1}dt}=\displaystyle\frac{(1-\varepsilon_i)L_i}{\int_0^{L_i}\sqrt{e^{\frac{2K}{n-1}}|x_i'(t)|^2+1}dt},\] where $ |x_i'(t)|^2\rightarrow 0$ uniformly by the above as $i\rightarrow \infty$, and $\varepsilon_i\rightarrow 0$.
\vspace{1em}

Thus, 

\begin{equation}\label{equation:distancegeqone}
\displaystyle\lim_{i\to\infty}\frac{d_\mathbb{R}(y(r_i),h(y(r_i)))}{L_i}\geq 1.
\end{equation}

\vspace{1em}

Since $y(t)$ and $h(y(t))$ are in $\mathbb{R}$, we can write $y(t)=\int_0^t y'(s)ds-y(0)$ and $h(y(t))=\int_0^t h_*(y'(s))ds-h(y(0))$. Also, the only possible isometries in $\mathbb{R}$ are reflections, translations, and a combination of the two. We want to show that $h_*$ cannot be a translation. 
\vspace{1em}

Suppose for the sake of contradiction that $h_*(y'(s))=y'(s)$. $$\displaystyle\frac{|h(y(r_i))-y(r_i)|}{L_i}=\frac{|\int_0^{r_i} y'(s)-\int_0^{r_i} y'(s)-h(y(0))+y(0)|}{L_i}=\frac{|h(y(0))-y(0)|}{L_i}.$$

Taking the limit of both sides, we get $\displaystyle\lim_{i\rightarrow\infty}\displaystyle\frac{|h(y(r_i))-y(r_i)|}{L_i}=0$, which is a contradiction. Thus, $h_*$ must be a reflection, and

\begin{equation}\label{equation:derivativesplit}
    h_*(\widetilde{\gamma}'(0))=-\widetilde{\gamma}'(0).
\end{equation}

\vspace{1em}

In order to show that $\Tilde{\gamma}$ is in the split direction, along with showing (\ref{equation:derivativesplit}), we must also show that $|x'(s)|=0$ for all $s$. We proceed by using (\ref{equation:derivativesplit}) to show that $\displaystyle\lim_{i\rightarrow\infty}\frac{2\int_0^{r_i}|y'(s)|ds}{L_i}=1$.
\vspace{1em}

By (\ref{equation:derivativesplit}), we have the following equality:\\

$\displaystyle\frac{2|\int_0^{r_i}y'(s)ds|}{L_i}=\displaystyle\frac{|\int_0^{r_i}y'(s)ds-\int_0^{r_i}h_*(y'(s))ds|}{L_i}$.\\

By the Fundamental Theorem of Calculus and the Triangle Inequality,\\

$=\displaystyle\frac{|y(r_i)-y(0)-h(y(r_i))+h(y(0))|}{L_i}\geq \displaystyle\frac{|h(y(r_i))-y(r_i)|}{L_i}-\displaystyle\frac{|h(y(0))-y(0)|}{L_i}$.\\

Taking the limit of both sides, and by (\ref{equation:distancegeqone}),\\

$\displaystyle\lim_{i\rightarrow\infty}\displaystyle\frac{2|\int_0^{r_i}y'(s)ds|}{L_i}\geq 1$.\\

On the other hand, since $|y'(s)|=1-e^{\frac{2f(s)}{n-1}}|x'(s)|\leq 1$,\\

$\displaystyle\lim_{i\rightarrow\infty}\frac{2\int_0^{r_i}|y'(s)|ds}{L_i}\leq \displaystyle\lim_{i\rightarrow\infty}\frac{2r_i}{L_i}=1$. This equality comes from \cite[Note~2.1]{So}.\\

Hence, $|y'(s)|=1$, so $|x'(s)|=0$, $\Tilde{\gamma}(t)=(x(0),y(t))$, and $\Tilde{\gamma}$ is in the split direction. 
\end{proof}

\begin{figure}[h]
    \centering
    \includegraphics[width=0.26\textwidth, height=0.27\textheight]{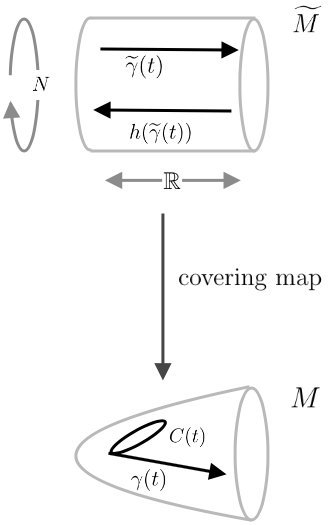}
    \caption{$C(t)$ is a representation of $h$ based at $\gamma(0)$. If $M$ satisfies the assumptions in Lemma \ref{lemma: mainresult}, then $\widetilde{\gamma}$ is in the split direction and $h(\widetilde{\gamma})$ is also in the split direction but facing the opposite direction of $\widetilde{\gamma}$.}
    \label{fig:mainlemma}
\end{figure}
\vspace{1em}

\begin{corollary}\label{corollary:flatnormalbundle}
If $M^n$ is a complete noncompact manifold with $\ric_\phi^1\geq 0$, $|\phi|$
bounded, and there exists an element $h\in\pi_1(M)$ which doesn't satisfy the loops to infinity property along a given ray $\gamma$, then $M^n$ is a flat normal bundle over a compact totally geodesic soul. 
\end{corollary}

We are now ready to prove Theorem \ref{thm:2} in the non-orientable case.

\begin{corollary}\label{cor:integralhomologybutnon-orientable}
 Let $M^n$ be a complete non-orientable Riemannian manifold and suppose one of the following holds:
\begin{enumerate}
    \item $\ric_X^N\geq 0$ with $N>n$.
    \item $\ric_\phi^N\geq 0$ with $N=\infty$, $\phi$ bounded above.
    \item $\ric_\phi^N\geq 0$ with $N\leq 1$ and $\phi$ bounded above.
    \item $\ric_\phi^N\geq 0$ with $N=\infty$, $\nabla\phi\rightarrow 0$ at $\infty$.
\end{enumerate}
Then $H_{n-1}(M,\mathbb{Z})=0$ or $\mathbb{Z}$.
\end{corollary}

\begin{proof}

We will only prove the $\ric_\phi^N\geq 0$ with $N\leq 1$ and $\phi$ bounded above case because the other cases follow similarly to \cite{ShSo}.
\vspace{1em}

Suppose $M$ is a two-ended manifold. Then by the Cheeger-Gromoll Splitting Theorem, $M$ splits isometrically as $L\times \mathbb{R}$ where $L$ is compact and has the same orientability as $M$. Then, since $M$ is non-orientable, $H_{n-1}(M,\mathbb{Z})=H_{n-1}(N,\mathbb{Z})=0$.
\vspace{1em}

Suppose $M$ is a one-ended manifold. Suppose $M^n$ satisfies loops to infinity property. $M$ has a double cover $\pi: \widetilde{M} \to M$ such
that $\widetilde{M}$ is orientable. We first claim that $\widetilde{M}$ has only one end.
\vspace{1em}

Assume for the sake of contradiction that $\widetilde{M}$ has two or more ends. By \cite[Lemma~4.4]{Wy}, either $\widetilde{M}$ splits isometrically as $\widetilde{M} = N^{n-1} \times \mathbb{R}$ where $N$ is compact, in which case we follow the proof of \cite[Propostion~3.2]{ShSo}, or $\widetilde{M}=N\times \mathbb{R}$ with $g=e^{\frac{2\phi}{n-1}}g_N+dr^2$ where $N$ contains no lines. Since $\widetilde{M}$ is orientable, so is the totally geodesic submanifold,
$N^{n-1}$. 
\vspace{1em}

Noting that $G(\widetilde{M})=\bigslant{\mathbb{Z}}{2\mathbb{Z}}$, let $h$ be the nontrivial deck transformation acting on $\widetilde{M}$ (i.e. $h_{\mathbb{R}}(r)\neq r$). By Theorem \ref{thm: wylabel2}, $h=h_\mathbb{R}\times h_N$, where $h_\mathbb{R}:\mathbb{R}\rightarrow \mathbb{R}$ and $h_N:N\rightarrow N$. Since $h_\mathbb{R}$ is an isometry, $h_\mathbb{R}(r)=\pm r+r_0$. If $h_\mathbb{R}(r)=r+r_0$, then $h_\mathbb{R}^2(r)=r+2r_0$. Since $h_\mathbb{R}^2(r)=r$, this implies that $r_0=0$, so $h_\mathbb{R}(r)=r$, which is a contradiction. Hence, $h_\mathbb{R}(r)=-r+r_0$.
\vspace{1em}

Now, we can use the topology of $M$ to show that $\widetilde{M}$ is one-ended and has the loops to infinity property. The interested reader can look at \cite[Proposition~3.2]{ShSo} for more details. By \cite[Proposition~2.1]{ShSo} , $H_{n-1}(\widetilde{M}, G)$ is trivial. Using the Universal Coefficient Theorem \cite[Theorem~55.1]{Mu}, $H_{n-1}(M,\mathbb{Z})=0$.
\vspace{1em}

Suppose $M^n$ be a one-ended and doesn't have a ray with the loops to infinity property. Since $M^n$ doesn't have a ray with loops to infinity property, by Corollary \ref{corollary:flatnormalbundle}, $M^n$ is a flat normal bundle over a compact totally geodesic soul. Since $M^n$ is one-ended, $N$ is orientable if and only if $M$ is non-orientable, so $H_{n-1}(M,G)=H_{n-1}(N,G)=\mathbb{Z}$.
\end{proof}

Next, we prove Theorem \ref{thm:main theorem with Abelian group}, which generalizes Theorem \ref{thm:2} to classify the $n-1$ homologies with coefficients in Abelian groups of spaces with nonnegative $N$-Bakry \'Emery Ricci curvature. This is the $N$-Bakry \'Emery Ricci curvature analog of Shen-Sormani's \cite[Theorem~1.1]{ShSo} and can be proved in the same way as their theorem except with Theorem \ref{thm: wylabel2} instead of the Cheeger-Gromoll Splitting Theorem. We give a sketch of the proof below. 
\vspace{1em}

\begin{proof}[Proof of Theorem \ref{thm:main theorem with Abelian group}]
Consider $M$ with two or more ends. If $\ric_\phi^N\geq 0$ with $N= 1$ and $\phi$ bounded above, then $M$ splits as $N\times \mathbb{R}$ as in Theorem \ref{thm: wylabel2}. If $N$ is orientable, then $H_{n-1}(M,G)$ is $G$, and if $N^{n-1}$ is not orientable, then $H_{n-1}(M,G)$ is $\ker(G\overset{\times2}\to G)$. In all other cases when $M$ has two or more ends, we use the Cheeger-Gromoll Splitting Theorem instead of Theorem \ref{thm: wylabel2}, as in \cite[Proposition~3.1]{ShSo} to get the same conclusion.
\vspace{1em}

If $M$ is one-ended with the loops to infinity property, then using Poincare Duality, Universal Coefficient Theorem, and other topological arguments, we get the desired result. Since this proof only uses topology, the proof is the same as \cite[Proposition~2.1]{ShSo}.
\vspace{1em}

Suppose $M$ is one-ended and doesn't have a ray with the loops to infinity property. If $\ric_\phi^N\geq 0$ with $\phi$ bounded above, then by Corollary \ref{corollary:flatnormalbundle}, $M^n$ is a flat normal bundle over a compact totally geodesic soul, $N^{n-1}$. In all other cases where $\ric_X^N\geq 0$ and $N$ is not 1, we use \cite[Theorem~1.2]{So} to get that $M^n$ is a flat normal bundle over a compact totally geodesic soul, $N^{n-1}$. Then, using that $M$ is one-ended, we get the desired conclusion. This proof is the same as \cite[Proposition~3.3]{ShSo}, except we use Corollary \ref{corollary:flatnormalbundle} in the $N=1$ case.
\end{proof}

\section*{Acknowledgements}
The author would like to thank her thesis advisor, Professor William Wylie, for his guidance, patience, and many helpful insights.\\

The author would also like to thank Professor Christina Sormani for helpful discussions.\\

This work was partially supported by NSF grant DMS-1654034. 

\bibliographystyle{apa}
\bibliography{sample}

\end{document}